\documentclass[12pt]{article}
\usepackage{}
\usepackage[numbers,sort&compress]{natbib}

\usepackage{color}
\usepackage{threeparttable}

\usepackage[english]{babel}
\usepackage{t1enc}
\usepackage{amsthm,amsmath,amssymb}
\usepackage{mathrsfs}
\usepackage[mathscr]{eucal}
\usepackage[latin1]{inputenc}
\usepackage[english]{babel}
\usepackage{latexsym}
\usepackage{graphicx}
\usepackage[noend]{algpseudocode}
\usepackage{algorithmicx,algorithm}
\usepackage{lineno}
\usepackage{booktabs}
\usepackage{multirow}
\newtheorem{theorem}{Theorem}
\newtheorem{corollary}{Corollary}

\newtheorem{lemma}{Lemma}
\newtheorem{remark}{Remark}
\newtheorem{conjecture}{Conjecture}

\newtheorem{definition}{Definition}

\def \T{\textup{T}}

\def \diag{\textup{diag~}}

\makeatletter
\newcommand{\rmnum}[1]{\romannumeral #1}
\newcommand\restr[2]{{
		\left.\kern-\nulldelimiterspace 
		#1 
		\right|_{#2} 
}}
\newcommand{\Rmnum}[1]{\expandafter\@slowromancap\romannumeral #1@}
 
\makeatother
\title{Disproof of a conjecture on the main spectrum of generalized Bethe trees}
\author{\small Zhidan Yan\quad\quad Wei Wang\thanks{Corresponding author: wangwei.math@gmail.com}
\\
{\footnotesize School of Mathematics, Physics and Finance, Anhui Polytechnic University, Wuhu 241000, P. R. China}
}
\date{}
\begin{document}
 \maketitle

\begin{abstract}
	An eigenvalue of the adjacency matrix of a graph is said to be main if the all-ones vector is not orthogonal to its associated eigenspace. A generalized Bethe tree with $k$ levels is a rooted tree in which vertices at the same level have the same degree. Fran\c{c}a and Brondani [On the main spectrum of generalized Bethe trees, Linear Algebra Appl., 628 (2021) 56-71] recently conjectured that any generalized Bethe tree with $k$ levels has exactly $k$ main eigenvalues whenever $k$ is even. We disprove the conjecture by constructing a family of counterexamples for  even integers $k\ge 6$.
 \\


\end{abstract}
\section{Introduction}
\label{intro}
Let $G$ be a simple graph with vertex set $\{1,\ldots,n\}$.  The \emph{adjacency matrix} of $G$ is the $n\times n$ symmetric matrix $A=A(G)=(a_{i,j})$, where $a_{i,j}=1$ if $i$ and $j$ are adjacent;  $a_{i,j}=0$ otherwise. We often identify a graph $G$ with its adjacency matrix $A$. For example, the eigenvalues and the eigenspaces of $A$ are called the eigenvalues and the eigenspaces of $G$. Let $e_n$ denote the $n$-dimensional all-ones vector. An eigenvalue of $G$ is called a \emph{main eigenvalue}
if the associated eigenspace is not orthogonal to  $e_n$. The \emph{main spectrum} of $G$ is the set of all (distinct) main eigenvalues of $G$. The notion of main eigenvalue was introduced by Cvetkovi\'{c} \cite{cvetkovic1978} and has received considerable attention since then; see \cite{rowlinson2007} for a survey.

Let $k\ge 2$ and $d_1,d_{2},\ldots, d_{k-1}$ be $k-1$ integers of at least 2. A \emph{generalized Bethe tree} $\mathcal{B}(d_1,d_{2},\ldots,d_{k-1})$ with $k$ levels is a rooted tree in which all vertices at the  level $i$ have the degree $d_{i}$ for $i=1,2,\ldots,k-1$. If $d_2=d_3=\cdots=d_{k-1}=d+1$ and $d_1=d$, then the generalized Bethe tree $\mathcal{B}(d_1,d_{2},\ldots,d_{k-1})$ becomes an ordinary \emph{Bethe tree}, and is denoted by $\mathcal{B}_{d,k}$. If $d_1=d_2=\cdots=d_{k-1}=d$ then $\mathcal{B}(d_1,d_{2},\ldots,d_{k-1})$ is  denoted by $\mathcal{Q}_{d,k}$ and is called the \emph{quasi-regular complete tree} \cite{cvetkovi2008}. Figure 1 shows some examples of generalized Bethe trees.

 Note that a generalized Bethe tree with two levels is a star graph, which is known to have exactly two main eigenvalues. Fran\c{c}a and Brondani extended this result for both Bethe trees and quasi-regular complete trees.

\begin{theorem}[\cite{franca2021LAA}]\label{twotrees}
	For any $k\ge 3$ and $d\ge 2$, both $\mathcal{B}_{d,k}$ and $\mathcal{Q}_{d,k}$ have exactly $k$ main eigenvalues.
\end{theorem}

\begin{figure}
	\centering
	\includegraphics[height=3cm]{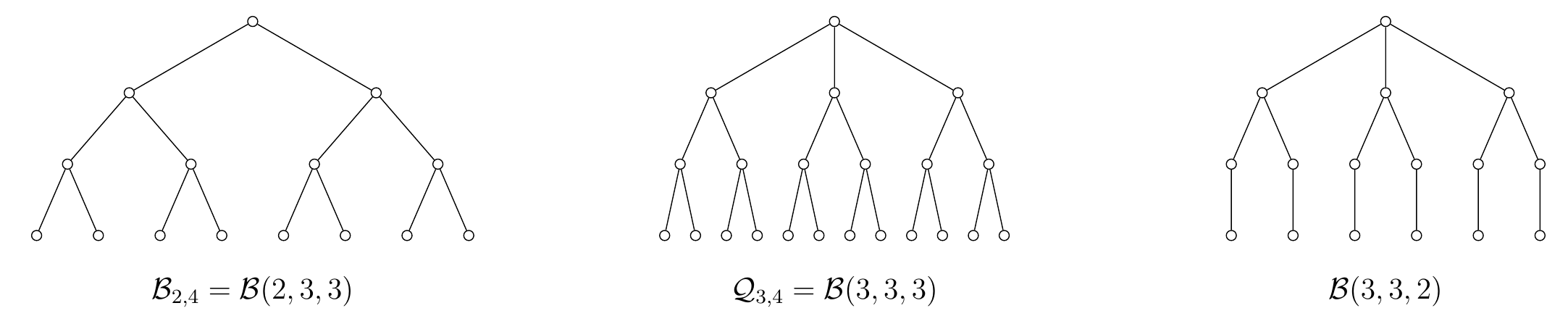}
	\caption{Generalized Bethe trees $\mathcal{B}_{2,4}$, $\mathcal{Q}_{3,4}$ and $\mathcal{B}(3,3,2)$.}
\end{figure}

It is known that Theorem \ref{twotrees} can not be  extended to all generalized Bethe trees. Indeed, for any positive integer $\alpha\ge 2$, the generalized Bethe tree $\mathcal{B}(\alpha^2-\alpha+1,\alpha)$ has exactly two main eigenvalues while it has three levels; see \cite{hou2005}. Fran\c{c}a and Brondani \cite{franca2021LAA} conjectured that such an inconsistency will never happen if the number of levels of the concerned generalized Bethe tree is even.
\begin{conjecture}[\cite{franca2021LAA}]\label{main}
Let $k$ be even. Every generalized Bethe tree with  $k$ levels has exactly $k$ main eigenvalues.
\end{conjecture}
Note that Conjecture \ref{main} trivially holds for $k=2$. The main aim of this note is to show that Conjecture \ref{main} fails for each even $k\ge 6$. Indeed, we give a family of counterexamples to  Conjecture \ref{main}.
\begin{theorem}\label{counttree}
Let $k\ge 6$ be even. Then $\mathcal{B}(5,k-3,5,3,\underbrace{2,2,\ldots,2}_{k-5})$ with $k$ levels has at most $k-1$ main eigenvalues.
\end{theorem}
The counterexample $\mathcal{B}(5,3,5,3,2)$ constructed in Theorem \ref{counttree} for $k=6$ is shown in Figure 2, where the solid circle  represents the root vertex.
\begin{figure}
	\centering
	\includegraphics[height=6cm]{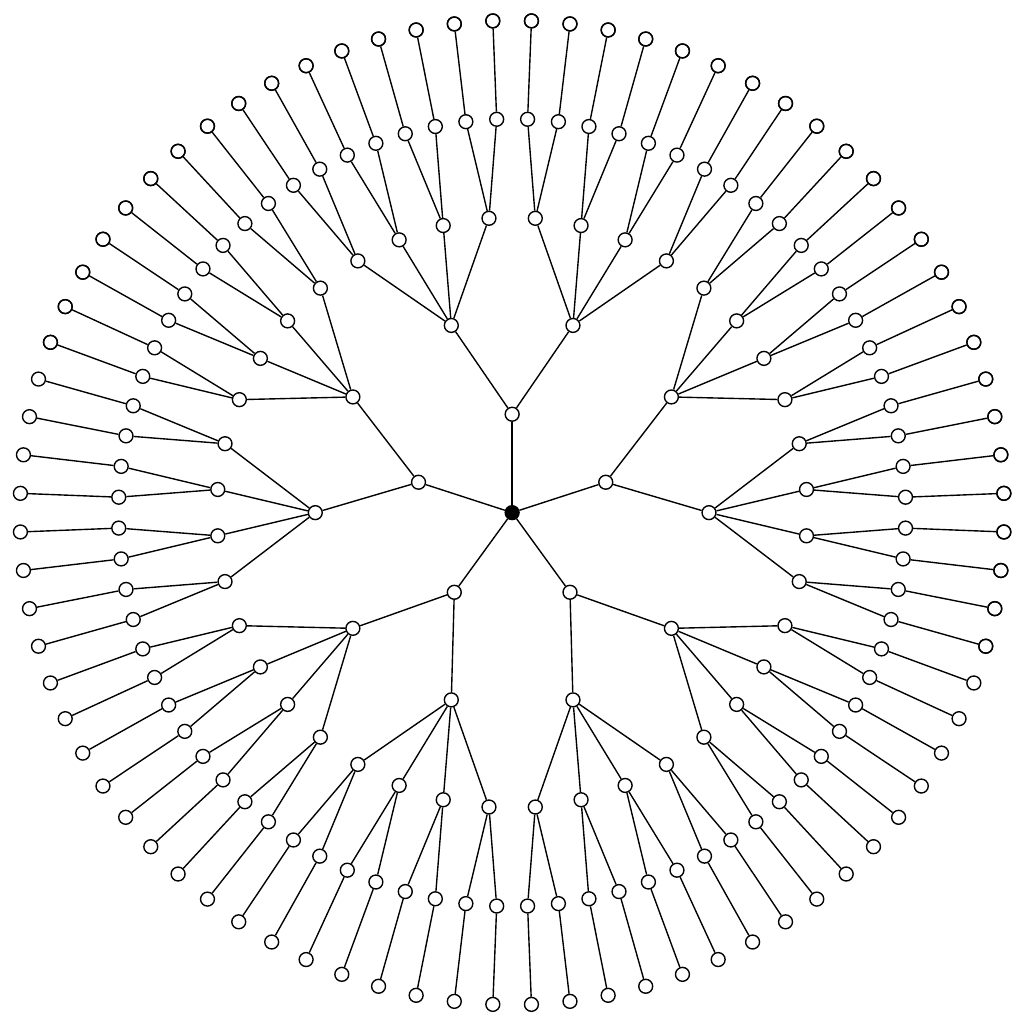}
	\caption{The counterexample $\mathcal{B}(5,3,5,3,2)$.}
\end{figure}
The proof of Theorem \ref{counttree} will be presented in the next section using the tool of equitable partitions. We remark that  Conjecture \ref{main} is still open for the only remaining case that $k=4$.

\section{Proof of Theorem \ref{counttree}}

Let $\Pi=\{C_1,C_2,\ldots,C_k\}$ be a partition of the vertex set $V(G)$. We say that $\Pi$ is an \emph{equitable partition} of $G$ if for all $i,j\in\{1,2,\ldots,k\}$, there is a constant $b_{ij}$ such that every vertex in $C_i$ has exactly $b_{ij}$ neighbors in $V_j$. The \emph{divisor} (or \emph{quotient graph}) of $G$ with respect to $\Pi$, denoted by $G/\Pi$, is the directed  multigraph with $k$ vertices $C_1,C_2,\ldots,C_k$ and $b_{ij}$ arcs from $C_i$ to $C_j$. The $k\times k$ matrix $(b_{ij})$ is called the \emph{divisor matrix} of $\Pi$, denoted by $A(G/\Pi)$. The \emph{characteristic matrix} of $\Pi$, denoted by $C$,  is the $n\times k$ matrix with the characteristic vectors of the cells $C_i$'s as its columns. Note that $C^\T C$ is the $k\times k$ diagonal matrix $\diag(|C_1|,|C_2|,\ldots,|C_k|)$, which is clearly invertible as each $|C_i|$ is nonzero.

The following lemma collects some basic results on divisors.
\begin{lemma}[\cite{cvetkovic2010}]\label{basic}
Let $\Pi=\{C_1, C_2,\ldots,C_k\}$ be an equitable partition of the graph $G$, with characteristic matrix $C$. Let $A=A(G)$ and $B=A(G/\Pi)$. Then\\
(\rmnum{1})   AC=CB;\\
(\rmnum{2})  $\det(xI-B)$ divides $\det(xI-A)$ and hence each eigenvalue of $B$ is an eigenvalue of $A$;
(\rmnum{3})  Each main eigenvalue of $A$ is an eigenvalue of $B$.
\end{lemma}
\begin{lemma}\label{BBt}
	Using the notations of Lemma \ref{basic}, and letting $D=C^\T C=\diag(|C_1|,|C_2|,\ldots,|C_k|)$, we have $D BD^{-1}=B^\T$.
\end{lemma}
\begin{proof}
	Let $b'_{ij}$ and $b_{ij}$ denote the $(i,j)$-th entry of $D BD^{-1}$ and $B$, respectively. Noting that $D^{-1}=\diag(\frac{1}{|C_1|},\frac{1}{|C_2|},\ldots,\frac{1}{|C_k|})$, we see that
		$$b'_{ij}=\frac{b_{ij}|C_i|}{|C_j|}.$$	
As $\Pi$ is an equitable partition, we see that $b_{ij}|C_i|$ equals $b_{ji}|C_{j}|$ since both count the number of edges with one end in $C_i$ and the other in $C_j$. Thus $b'_{ij}=b_{ji}$ and the lemma follows.
\end{proof}
\begin{corollary}\label{twoequiv}
	Using the notations of Lemma \ref{basic}, and letting $\lambda$ be an eigenvalue of $B$ (or equivalently, $B^\T$), the following two statements are equivalent:\\
	(\rmnum{1}) $B$ has an eigenvector associated with $\lambda$ which is not orthogonal to the vector $(|C_1|,\ldots,|C_k|)^\T$.\\
	(\rmnum{2}) $B^\T$ has an eigenvector associated with $\lambda$ which is not orthogonal to $e_k$.\\
\end{corollary}
\begin{proof}
	Let $\xi$ be a vector such that $B\xi=\lambda\xi$ and $(|C_1|,\ldots,|C_k|)\xi\neq 0$. Write  $D=C^\T C$. Then by Lemma \ref{BBt}, we see that $D B=B^\T D$ and hence $B^\T D \xi=D B\xi=\lambda D \xi$. Noting that $e_k^\T D=(|C_1|,\ldots,|C_k|)$, we have $e_k^\T(D \xi)=(|C_1|,\ldots,|C_k|)\xi\neq 0$. This means that $D \xi$ is an eigenvector of $B^\T$ associated with $\lambda$ which is not orthogonal to $e_k$. Thus,  (\rmnum{1}) implies (\rmnum{2}).
	
	Conversely, 	let $\eta$ be a vector such that $B^\T \eta=\lambda\eta$ and $e_k^\T\eta\neq 0$. Similarly, $BD^{-1}=D^{-1}B^\T$ and hence $D^{-1}\eta$ is an eigenvector of $B$ associated with $\lambda$. Moreover, we have $(|C_1|,\ldots,|C_k|)D^{-1}\eta=e_k^\T\eta\neq 0$. Thus (\rmnum{2}) implies (\rmnum{1}) and the proof is complete.
\end{proof}
The following definition was first introduced by Teranishi \cite{teranishi2001}.
\begin{definition}[\cite{teranishi2001}] \normalfont
	Let  $\Pi=\{C_1, C_2,\ldots,C_k\}$ be an equitable partition of the graph $G$. An eigenvalue $\lambda$ of the divisor matrix $A(G/\Pi)$ is called a \emph{main eigenvalue} of $G/\Pi$ if the associated eigenspace of $(A(G/\Pi))^\T$ is not orthogonal to the all-ones vector $e_k$.	
\end{definition}
\begin{remark}\normalfont
	The original definition of main eigenvalues of divisors given in \cite{teranishi2001} requires both conditions as described in Corollary \ref{twoequiv}. But as these two conditions are indeed equivalent, the current simplified definition is essentially equivalent to the original one.
\end{remark}
We need the following key lemma due to Teranishi \cite{teranishi2001}, which  refines the last two assertions of Lemma \ref{basic}. We give a simpler proof for the convenience of readers.
\begin{lemma}[\cite{teranishi2001}] \label{equalmain}
	Let $G$ be a graph with an equitable partition $\Pi=\{C_1,\ldots,C_k\}$. Then $G$ and $G/\Pi$ have the same main spectrum.
\end{lemma}
\begin{proof}
	Let $A=A(G)$, $B=A(G/\Pi)$ and $C$ be the characteristic matrix of $\Pi$. By Lemma \ref{basic}(\rmnum{1}), we have
	\begin{equation}\label{basiceq}
AC=CB.
	\end{equation}
	Let $\lambda$ be a main eigenvalue of $G$.  Then there exists a vector $\xi\in \mathbb{R}^{n}$ such that $A \xi=\lambda\xi$ and $e_n^\T \xi\neq 0$. Taking transpose for both sides of \eqref{basiceq} and noting that $A$ is symmetric, we have $C^\T A=B^\T C^\T$ and hence $B^\T C^\T\xi=C^\T A\xi=\lambda C^\T\xi$. Write $\eta=C^\T\xi$. Then $B^\T \eta=\lambda \eta$. Moreover, as $Ce_k=e_n$ and $e_n^\T \xi\neq 0$, we have
	 	$$e_k^\T \eta =e_k^\T (C^\T \xi)=(Ce_k)^\T \xi=e_n^\T \xi\neq 0.$$
This indicates  that $\lambda$ is a main eigenvalue of $G/\Pi$.

	Conversely, let $\lambda$ be a main eigenvalue of $G/\Pi$. By Corollary  \ref{twoequiv}, there exists a $\eta\in\mathbb{R}^k$ such that $B\eta =\lambda \eta$ and $(|C_1|,\ldots, |C_k|)^\T \eta\neq 0$. Write $\xi=C\eta$. It follows from \eqref{basiceq} that
	$$A\xi=AC\eta =CB\eta=\lambda C\eta=\lambda\xi.$$
Moreover, as $e^\T_n C=(|C_1|,\ldots,|C_k|)$, we have
	$$e_n^\T \xi=e_n^\T C\eta=(|C_1|,\ldots,|C_k|)\eta\neq 0.$$
	This proves that $\lambda$ is a main eigenvalue of $G$. The proof is complete.
\end{proof}
Now we are ready to prove Theorem \ref{counttree}. Let  $G_k=\mathcal{B}(5,k-3,5,3,\underbrace{2,2,\ldots,2}_{k-5})$, where $k$ is even and $k\ge 6$. Let $\Pi_k=\{C_1,C_2,\ldots,C_k\}$ be a partition $V(G_k)$, where $C_i$ collects all vertices at the $i$-th level of the generalized Bethe tree $G_k$. Clearly, $\Pi_k$ is an equitable partition and the corresponding divisor matrix $B$ is the following tridiagonal matrix:

$$ B=\small{\left(
	\begin{array}{ccccccccc}
	0 &  5 & {} & {} &{} &{}  &{}  &{} & {} \\
	1 & 0 & k-4 & {} &{} & {} & {} & {}& {}\\
	{} & 1 & 0 & 4  & {} &  {}&  {}&{} & {}  \\
	{}  &{} & 1 & 0 & 2 &{}  & {} &{} & {} \\
	{}&{}  & {} & 1 & 0 & 1 &{}  & {}&{}  \\
	{}& {} & {} &{}  & 1 & 0 &\ddots  & {} & {}\\
	{}&{}  & {} & {} & {}&\ddots & \ddots & \ddots & {} \\
    {} & {} & {} & {} &{} & {}&\ddots & \ddots & 1 \\
  {}&{}&{} &{} & {} & {} &{}  & 1 & 0 \\
	\end{array}
	\right)_{k\times k}}.
$$
Let $\xi=(1,-2,-1,2(k-3),\underbrace{-4(k-4),4(k-5),-4(k-6),\ldots,-8,4}_{k-4})^\T$. Note that $k$ is even and  the last $k-4$ entries of $\xi$ have alternating signs. Direct calculation shows that
\begin{equation*}
(2I+B^\T)\xi=\small{\left(
	\begin{array}{ccccccccc}
	2 &  1 & {} & {} &{} &{}  &{}  &{} & {} \\
    5 & 2 & 1 & {} &{} & {} & {} & {}& {}\\
    {} & k-4 & 2 & 1  & {} &  {}&  {}&{} & {}  \\
    {}  &{} & 4 & 2 & 1 &{}  & {} &{} & {} \\
    {}&{}  & {} & 2 & 2 & 1 &{}  & {}&{}  \\
    {}& {} & {} &{}  & 1 & 2 &\ddots  & {} & {}\\
    {}&{}  & {} & {} & {}&\ddots & \ddots & \ddots & {} \\
    {} & {} & {} & {} &{} & {}&\ddots & \ddots & 1 \\
   {}&{}&{} &{} & {} & {} &{}  & 1 & 2 \\
	\end{array}
	\right)}\normalsize{\left(
	\begin{array}{c}
1\\-2\\-1\\2(k-3)\\-4(k-4)\\4(k-5)\\\vdots\\-8\\4
	\end{array}
	\right)=\left(
	\begin{array}{c}
	0\\0\\0\\0\\0\\0\\\vdots\\0\\0
	\end{array}
	\right)},
	\end{equation*}
	that is, $B^\T \xi=-2\xi$.
This indicates that $\lambda=-2$ is an eigenvalue of $B^\T$.
Moreover, as the sum of the last $k-4$ entries of $\xi$ is $-2(k-4)$, we see that
\begin{equation}\label{exi}
e_k^\T \xi=1-2-1+2(k-3)-2(k-4)=0.
\end{equation}
Note that $G_k/\Pi_k$ has exactly $k$ vertices.  We claim that the divisor $G_k/\Pi_k$ has at most $k-1$ main eigenvalues. Otherwise all the $k$ eigenvalues of $G_k/\Pi_k$ must be simple and main. Consider the eigenvalue $\lambda=-2$. Then the associated eigenspace is one-dimensional and is spanned by $\xi$. Thus, by \eqref{exi}, we see that $\lambda=-2$ is not a main eigenvalue, a contradiction.

By Lemma \ref{equalmain}, we know that $G_k$ and $G_k/\Pi_k$ have the same number of main eigenvalues. Thus, $G_k$ has at most $k-1$ main eigenvalues. This completes the proof of Theorem \ref{counttree}.
\section*{Declaration of competing interest}
There is no conflict of interest.
\section*{Acknowledgments}
This work is supported by the	National Natural Science Foundation of China (Grant Nos. 12001006 and 11971406) and the Scientific Research Foundation of Anhui Polytechnic University (Grant No.\,2019YQQ024).

\end{document}